\documentclass[a4paper,12pt]{article}

\usepackage{amsfonts,amsthm,amsmath,amssymb,graphicx}

\title{Projections of the natural measure for percolation fractals}

\date{}

\author{Yuval Peres\\ \normalsize Microsoft Research, Redmond \and Micha\l \ Rams \thanks{partially supported by the MNiSW grant N201 607640 (Poland)}\\
\normalsize Institute of Mathematics, Polish Academy of Sciences}

\theoremstyle{plain}
\newtheorem{lem}{Lemma}[section]
\newtheorem{prop}[lem]{Proposition}
\newtheorem{thm}[lem]{Theorem}
\newtheorem{cor}[lem]{Corollary}

\theoremstyle{definition}

\theoremstyle{remark}
\newtheorem*{rem}{Remark}

\numberwithin{equation}{section}

\renewcommand{\epsilon}{\varepsilon}

\newcommand{\ien}{\underline{i}_n}
\newcommand{\jen}{\underline{j}_n}

\begin{document}

\maketitle

\def\thefootnote{}
\def\thefootnote{\arabic{footnote}}
\setcounter{footnote}{0}

\begin{abstract}
We prove that, with probability 1, all orthogonal projections of the
natural measure on a percolation fractal are absolutely continuous
and (except for  the horizontal and vertical projection) have
H\"older continuous density.
\end{abstract}

\section{Introduction}

The objects of study in the present paper are percolation fractals
in the plane and their properties under projections. Percolation
fractals are an important class of random fractals, introduced by
Mandelbrot in \cite{M}; we refer the reader to \cite{C} or
\cite{G} for their properties and further references. Of special
importance for us will be the natural measure on a percolation
fractal, defined by Mauldin and Williams in \cite{MW}. It is a
random probability measure, almost surely
supported on the percolation fractal.

Our goal is to study the projection properties of the natural 
measure, in the sense of Marstrand Theorem. It is a continuation
of the work in \cite{RS}, where the projection properties of the
percolation fractal itself were studied. The main result of
\cite{RS} was that if the expected Hausdorff dimension of the
percolation fractal was greater than 1 then for almost all
realizations, all linear projections of the fractal contained an
interval.

Our main result is that, under the same assumptions, for almost
all realizations, all linear projections of the natural measure
are absolutely continuous and almost all (except the horizontal
and vertical projections) have H\"older continuous density. The density of the projection in horizontal or vertical direction is almost surely H\"older continuous in the $k$-symbolic metric but is in general discontinuous at all the $k$-adic points. While
this result implies the result from \cite{RS}, it should be
mentioned that the approach in \cite{RS} was robust and applicable
to certain modified random fractals, while the approach in the
present paper works for percolation fractals only.

\section{Notation and results}

Let us begin by recalling the construction of percolation fractal.
There will be two parameters: an integer $k\geq 2$ and a real
number $p\in (0,1)$. Let
$K=[0,1]^2=\bigcup _{i,j=0}^{k-1}
K_{i,j}$, where
$K_{i,j}:=\left[\frac{i}{k},\frac{i+1}{k}\right]\times
\left[\frac{j}{k},\frac{j+1}{k}\right]$. In general, for
$\ien:=(i_1,\dots ,i_n)\in \{0,\ldots,k-1\}^n$ and
$\jen:=(j_1,\dots ,j_n)\in \{0,\ldots,k-1\}^n$ we write

\[
I_{\ien}:=\left[\sum_{\ell=1}^{n}i_\ell\cdot
k^{-\ell},k^{-n}+\sum_{\ell=1}^{n}i_\ell\cdot k^{-\ell}\right]
\]
and

\[
K_{\ien,\jen} = I_{\ien} \times I_{\jen}.
\]

Let $\Sigma_n=\{0,\ldots,k-1\}^n\times\{0,\ldots,k-1\}^n$. We
construct a family of random sets $\mathcal{E}_n(\omega)\subset
\Sigma_n$ in the following way. We begin with

\[
P(( i, j) \in \mathcal{E}_1) = p,
\]
(with all these events mutually independent)
and then we continue inductively:
\[
P((\ien i, \jen j) \in \mathcal{E}_{n+1}|
(\ien,\jen)\in \mathcal{E}_n) = p,
\]
(with these events mutually independent) and $(\ien i,
\jen j) \notin \mathcal{E}_{n+1}$ if $(\ien,\jen)\notin
\mathcal{E}_n$. We write

\[
E_n=\bigcup_{(\ien,\jen)\in \mathcal{E}_n} K_{\ien,\jen}
\]
and

\[
E=\lim_{n\to\infty} E_n =\bigcap_n E_n.
\]

Thus, $E$ is a random fractal set (which we will call the {\it
realization} of a percolation fractal).

Note that ${\cal E}_n$ is the $n$-th step in a branching process which
has an average number of $k^2p$ children for a parent. Hence, as
long as $p>1/k^2$, there is a positive
probability that $E$ is nonempty.

The limit

\[
Z(E)=\lim_{n\to\infty} (k^2 p)^{-n}\cdot \sharp \mathcal{E}_n
\]
exists and is finite almost surely by \cite{AN}, and it is almost
surely positive if $E$ is nonempty. Moreover, by \cite{F} and
\cite{MW}, for almost all (nonempty) realizations of $E$

\[
\dim_H E = \frac {\log {k^2 p}} {\log k}.
\]
In what follows, we will always assume that $kp>1$, which implies that the
Hausdorff dimension of the percolation fractal (when nonempty) is greater than 1 almost surely.

The {\it natural measure} is the weak limit

\[
\mu=\lim_{n\to\infty} \frac 1 {p^n Z(E)} \rm{Leb}|_{E_n},
\]
it exists almost surely, see \cite{MW}.

For $\theta\in [0,\pi)$, we will consider the  projections $\pi_\theta$ in direction $\theta$ defined  by $\pi_\theta(x,y)=x\cos \theta+ y \sin \theta$. The image of the measure $\mu$ under the map $\pi$ will be denoted $\pi^* \mu$. Our main result is the following:

\begin{thm} \label{thm:main}
Assume $k p>1$. If $E$ is nonempty then almost surely all the
projections $\mu^\theta=\pi_\theta^* \mu$ are absolutely continuous
with respect to the Lebesgue measure. Moreover, almost surely the
density of $\mu^\theta$ is H\"older continuous for $\theta
\neq 0,\pi/2$. For the horizontal and vertical projections the density of the projected measure will in general be undefined at the $k$-adic points, but it will almost surely be H\"older continuous in the metric
\[
\rho(x,y) = \exp(-\log k \cdot \min \{\ell : \exists m\, x<mk^{-\ell}<y\})
\]
everywhere except at the $k$-adic points.
\end{thm}

\begin{rem}
The H\"older exponent can be chosen independent of $\theta$, the H\"older constant blows up as $\theta$ approaches horizontal or vertical direction. We thank Pablo Shmerkin for this observation.
\end{rem}


In the next section we present the method we use to estimate the density of the projected measure. A similar approach was first applied in \cite{FG}. In the fourth section we consider horizontal and vertical projections. Finally, in the last section we present the proof for non-horizontal and non-vertical directions.


\section{Large deviation estimates for the projected measure}

To study $\mu$ and $\mu^\theta$, we will use their level $n$ approximations. We will denote

\[
\mu_n= \frac 1 {p^n} \rm{Leb}|_{E_n}
\]
and

\[
\mu_n^\theta = \pi_\theta^* \mu_n.
\]

We will denote $\ell_x^\theta = \pi_\theta^{-1}(x)$.

Note that the weak limit $\tilde{\mu}=\lim \mu_n$ is $Z(E) \mu$, not $\mu$.
We choose this normalization because the sequence $\{\mu_n\}$ of random
measures has an important martingale property:

\begin{equation} \label{martingale}
E(\mu_{n+1}| \mathcal{E}_n) = \mu_n.
\end{equation}

The measure $\mu_n^\theta$ is obviously absolutely continuous with
respect to Lebesgue measure; we denote its density by
$y_n^\theta$. Analogous to \eqref{martingale}, we have

\begin{equation} \label{martingale2}
E(y_{n+1}^\theta(x)| \mathcal{E}_n) = y_n^\theta(x).
\end{equation}

Note the geometric interpretation:
\[
y_n^\theta(x) = p^{-n} \cdot |\ell_x^\theta \cap E_n|.
\]

Our goal in this section: assuming we know $y_n^\theta(x)$ for some fixed $\theta$ and $x$, we would like to estimate $y_{n+1}^\theta(x)$. For every $K_{\ien, \jen}\subset E_n$ that the line $\ell_x^\theta$ intersects, it can intersect between 0 and $2k-1$ of its children (each of $n+1$-level subsquares of $K_{\ien,\jen}$ intersecting this line is in $E_{n+1}$ with probability $p$, independently). Given $(\ien, \jen)\in {\cal E}_n$, let $Y(\ien, \jen; x,\theta)$ be the length of intersection of $\ell_x^\theta$ with the union of squares $K_{\ien i, \jen j}$ for those $i,j$ for which $(\ien i, \jen j)\in {\cal E}_{n+1}$.

Naturally, we have

\begin{equation} \label{eqn:y}
y_{n+1}^\theta(x) = p^{-n-1}\cdot \sum_{(\ien, \jen)\in{\cal E}_n} Y(\ien, \jen; x,\theta).
\end{equation}

The random variables $Y(\ien, \jen; x,\theta)$ are conditionally independent given $E_n$; they take values between $0$ and $\sqrt 2 k^{-n}$ and satisfy

\[
E(Y(\ien, \jen; x,\theta)|{\cal E}_n) = p\cdot|\ell_x^\theta\cap K_{\ien, \jen}|.
\]


Given $y_n^\theta(x)$, we will estimate $y_{n+1}^\theta(x)$. We are going to use the following result (a variation of Hoeffding inequality). We denote $||X_i|| = \sup_\omega |X_i|(\omega)$.

\begin{lem} \label{lem:ld}
Let $\{X_i\}$ be a family of independent bounded random variables
with $E(X_i)=0$ and $||X_i|| \leq 1$. Set $S=\sum X_i$ and $\Upsilon=\sum
||X_i||$. Then for any positive $a$ we have

\[
P(S>a) \le  e^{-a^2/2\Upsilon}.
\]

\begin{proof}
It is an immediate consequence of \cite[Theorem 2]{H}.
\end{proof}
\end{lem}

Applied to our particular sum of random variables, Lemma \ref{lem:ld} implies the following:

\begin{lem} \label{lem:cre}
There exist $C_1>0$ and $\gamma<1$ such that the following statements
are true.
\begin{itemize}
\item[i)] If $x, \theta, {\cal E}_n$ satisfy $y_n^\theta(x) > 1$ then
\[
P(y_{n+1}^\theta(x) < y_n^\theta(x) + p^{-n} k^{-n} (p^n k^n
y_n^\theta(x))^{2/3}|{\cal E}_n) > 1 - C_1 \gamma^{(pk)^{n/3}}.
\]

\item[ii)] If $x, \theta, {\cal E}_n$ satisfy $y_n^\theta(x) < (pk)^{n/3}$ then
\[
P(|y_{n+1}^\theta(x) - y_n^\theta(x)| < (pk)^{-n/6}|{\cal E}_n) > 1 - C_1 \gamma^{(pk)^{n/3}}.
\]

\end{itemize}
\end{lem}
\begin{proof}
 We will apply Lemma \ref{lem:ld} to the random variables $X_{\ien, \jen}=k^n (Y(\ien, \jen; x,\theta)-p|\ell_x^\theta\cap K_{\ien, \jen}|)/\sqrt{2}$. We have

\[
||X_{\ien, \jen}|| =  \frac 1 {\sqrt{2}} \max(p, 1-p) k^n |\ell_x^\theta\cap K_{\ien, \jen}| \le 1,
\]
whence
\[
\Upsilon = \frac 1 {\sqrt{2}} \max(p, 1-p) p^n k^n y_n^\theta(x).
\]


To prove statement i), we choose
\[
a = \frac 1 {\sqrt 2} p (p^n k^n y_n^\theta(x))^{2/3}.
\]
Denoting 
\[
\gamma = \exp\left(- \frac 1 {2\sqrt{2}} \frac {p^2} {\max(p,1-p)}\right)<1,
\]
we obtain that
\[
P\Bigl(y_{n+1}^\theta(x) \ge y_n^\theta(x) + p^{-n} k^{-n} (p^n k^n y_n^\theta(x))^{2/3} | {\cal E}_n \Bigr) \le 
\gamma^ {(pk)^{n/3} (y_n^\theta(x))^{1/3}}.
\]

Substituting $y_n^\theta(x) > 1$ gives

\[
P(y_{n+1}^\theta(x) \ge  y_n^\theta(x) + p^{-n} k^{-n} (p^n k^n y_n^\theta(x))^{2/3} | {\cal E}_n) \le \gamma^{(pk)^{n/3}} \,.
\]

The statement ii) follows in an analogous way by choosing

\[
a = \frac 1 {\sqrt 2} p (pk)^{5n/6}.
\]
\end{proof}

\section{Horizontal and vertical projections}

By symmetry, we only need to consider the vertical projection, $\theta=\pi/2$. Observe that $y_n^\theta(x)$ is constant on the open $k$-adic intervals of level $n$. We will write $I_{\ien}$ for the $k$-adic interval of 
length $k^{-n}$ with address $\ien$, and $y_n(\ien)$ for the value of $y_n^{\pi/2}(x)$ when $x\in I_{\ien}$.

We fix $C_1$ for which Lemma \ref{lem:cre} holds.
It has two immediate corollaries  we will need. 

The first corollary guarantees exponential speed of convergence of
$y_n^\theta(x)$ if for some $n$ it is not too big.

\begin{cor} \label{cor:small}
If $y_N(\underline{i}_N)< (pk)^{N/4}$ and $N>N_0$ then
\[
P(\forall n\geq N \ \forall x\in I_{\underline{i}_N}\ |y_{n+1}^\theta(x)-y_n^\theta(x)| < (pk)^{-n/6}|{\cal E}_N)
\geq 1 - C_1 \sum_{m=N}^\infty k^{m-N} \gamma^{(pk)^{m/3}}.
\]
\begin{proof}
The event in the assertion will be satisfied if the event from Lemma
\ref{lem:cre} ii) happens for all $n\geq N$ for all $k^{n-N}$
sequences $\ien$ beginning with $\underline{i}_N$ and for all
$l$. Note that in this situation, as $N>N_0$, \eqref{cond2}
automatically guarantees that $y_{n+1}(\ien i) < (pk)^{(n+1)/4}$,
hence the assumptions of Lemma \ref{lem:cre} ii) are satisfied for every $n$.
\end{proof}
\end{cor}

The second corollary we will prove
guarantees that at most $k$-adic intervals of high enough level $n$,
$y_n$ is not too big. Let $N_0$ be the smallest number for which

\begin{equation} \label{cond1}
1 + (pk)^{- N_0/3} < (pk)^{1/8}.
\end{equation}

In particular, we have
\begin{equation} \label{cond2}
1 + (pk)^{-5N_0/12} < (pk)^{1/4}.
\end{equation}

\begin{cor} \label{cor:big}
There exists $L>1$ such that for all $n>LN_0$ and for all $x$ we have
\[
P(y_n^\theta(x) < (pk)^{n/4}) > 1 - C_1 \sum_{m=n/L}^n \gamma^{(pk)^{m/3}}.
\]

\begin{proof}
Outline of the proof: there will be three time periods. For the first period, $m\in [0,N_0]$, we do not put any restrictions on $y_m^\theta(x)$. In the second period, $m\in [N_0, \ell_0]$, $y_m^\theta(x)$ will be large, but we will use Lemma \ref{lem:cre}i) to prove that, with large probability, $\frac 1m \log_{pk} y_m^\theta(x)$ will be decreasing, eventually decreasing below $1/4$. We set $\ell_0$ as the first time $m\geq N_0$ for which $y_m^\theta(x)<(pk)^{m/4}$. Note that it can happen that $m=N_0$, in such a situation we skip the second period and proceed immediately to the third one. In the third period, $m\geq \ell_0$, we simply apply Corollary \ref{cor:small}.

Let us start. Fix
\[
L= \lceil -8 \log_{pk} p\rceil +1\,.
\]
 
By definition, 

\[
y_{N_0}^\theta(x) \leq p^{-N_0},
\]
which is all we will need to know about the first period.

Assume that $y_{N_0}^\theta(x)> (pk)^{N_0/4}$ (otherwise we pass immediately to the investigation of the third period, below). As long as $y_m^\theta(x) > 1$, as $m\geq N_0$, \eqref{cond1} implies

\[
y_m^\theta(x) + p^{-m} k^{-m} (p^m k^m y_m^\theta(x))^{2/3} = y_m^\theta(x) (1+ (p^{m} k^{m} y_m^\theta(x))^{-1/3}) < (pk)^{1/8} y_m^\theta(x) \]

and hence, Lemma \ref{lem:cre} i) guarantees that (with probability at least $1-C_1 \gamma^{(pk)^{m/3}}$) \[ \log_{pk} y_{m+1}^\theta(x) < \log_{pk}
y_m^\theta(x) + \frac 1 8.
\]

Hence, if the event in Lemma \ref{lem:cre} i) holds each time (for each $m = N_0,\ldots$), we will have

\begin{equation} \label{qwert}
y_{\ell N_0}^\theta(x) < (pk)^{(\ell-1)N_0/8} p^{-N_0}.
\end{equation}

The right hand side of \eqref{qwert} grows only as fast as $(pk)^{m/8}$, hence $\frac 1m \log_{pk} y_m^\theta(x)$ will be decreasing and will eventually decrease below $1/4$. Let us denote the first $m$ where this occurs by $\ell_0$. As

\[
(pk)^{LN_0/4} > (pk)^{(L-1)N_0/8} p^{-N_0} \] by the definition of $L$, we get $\ell_0\leq L N_0$. 

We now start the third period, and the assertion will follow from Corollary \ref{cor:small}, we just need to estimate the relevant probability.
 Calculating the probability of the events in Lemma \ref{lem:cre}i) happening for each  $m\geq N_0$ and applying Corollary \ref{cor:small}, we get 

\begin{eqnarray*}
P\Bigl(y_{n}^\theta(x) < (pk)^{n/4} \ \forall n\geq LN_0\Bigr) &>& \\ \Big(1 - \sum_{m=N_0}^{\ell_0-1} C_1 \gamma^{(pk)^{m/3}}\Big) &\cdot& \Big(1 -
\sum_{m=\ell_0}^{\infty} C_1 \gamma^{(pk)^{m/3}}\Big)   >   \\
&& 1 -\sum_{m=N_0}^{\infty} C_1 \gamma^{(pk)^{m/3}}.
\end{eqnarray*}
\end{proof}
\end{cor}

The main result of this section is the following:

\begin{prop} \label{prop:hori}
There exists $b<1$ such that almost surely there exist $C_2>0$ such
that for all $x\in [0,1]$ except the $k$-adic points and for all
$N>0$ we have
\[
\left|y_{N}^\theta(x) - \lim_{n\to \infty} y_n^\theta(x) \right| < C_2 b^{N}.
\]
In particular, the limit exists at all points except possibly at the
$k$-adic points.
\end{prop}
\begin{proof}

Corollaries  \ref{cor:small} and \ref{cor:big} guarantee that for all
$N>LN_0$ and for any cylinder $I_{\underline{i}_N}$ for all
non-$k$-adic $x\in I_{\underline{i}_N}$ the probability that the sequence $y_n^\theta(x)$
converges to a limit $y(x)$ and that

\begin{equation} \label{eqn:cond}
|y(x) - y_N^\theta(x)|< \frac {1} {1-(pk)^{-1/6}} (pk)^{-N/6}
\end{equation}
is at least

\[
p_N= 1 - C_1 \sum_{m=N/L}^N \gamma^{(pk)^{m/3}} - C_1 \sum_{m=N}^\infty
k^{m-N} \gamma^{(pk)^{m/3}}.
\]
As we have

\[
\sum_{N=N_0}^\infty k^N (1-p_N) < \infty,
\]
by Borel-Cantelli Lemma \eqref{eqn:cond} almost surely holds for all
except finitely many $k$-adic intervals of level greater than $LN_0$,
hence it almost surely holds for all $k$-adic intervals of level
greater than some $N_1$. We are done.
\end{proof}

After we proved Proposition \ref{prop:hori}, the
horizontal/vertical projections part of Theorem \ref{thm:main}
follows easily. As the function $y_N^\theta$ is constant on the
$k$-adic intervals of level $N$, for any $x,y\in (lk^{-N},
(l+1)k^{-N})$ we have

\[
\left| \lim_{n\to \infty} y_n^\theta(x) - \lim_{n\to \infty}
  y_n^\theta(y) \right| < 2C_2 b^N.
\]
That is, the measure

\[
\lim_{n\to\infty} \mu_n^\theta = \pi_\theta^* \tilde{\mu}= \frac 1 {Z(E)} \pi_\theta^* \mu
\]
exists, is absolutely continuous with respect to  Lebesgue measure
and its density is H\"older continuous in the metric $\rho$.

\section{General case}

In this section we will consider the general case of Theorem
\ref{thm:main}, that is all the projections in directions
different from the horizontal or vertical one. It is enough to prove the assertion for $\theta\in
(0,\pi/2)$, other directions will follow by symmetry.

It will be convenient for us to assume that all $\pi_\theta$ have the same range, which will be denoted by $\Delta$. For example, we might replace $\pi_\theta$ with a linear projection in direction $\theta$ from $K$ to the interval $K\cap \{y=1-x\}$. Such a replacement will change the densities $y_n^\theta(x)$, but only by a bounded multiplicative constant. In particular, the conclusions of Lemma \ref{lem:cre} and Corollary \ref{cor:big} hold.

We will prove the following proposition.

\begin{prop} \label{prop:gen}
There exists $b<1$ such that
almost surely the following holds. For every $\delta>0$ there exist $C_3>0$ and $N_2>0$ such that for all $N>N_2$, for all pairs of points points $x,y\in \Delta, |x-y|<k^{-N-1}$, and
for all $\theta \in [\delta, \pi/2-\delta]$ we have

\[
\left| \lim_{m\to\infty} y_m^\theta(x) - \lim_{m\to\infty}
  y_m^\theta(y)\right| < C_3 b^N.
\]
In particular, the limits exist everywhere.
\begin{proof}

Comparing with Proposition \ref{prop:hori}, there are two main difficulties: $y_n^\theta(\cdot)$ is no longer locally constant and we need the statement for every $\theta$, not just for one direction. Our solution is to prove that $y_n^\theta(x)$ is Lipschitz (in $x$ and $\theta$) and then calculate $y_n^\theta(x)$ only for finite (increasing with $n$) families of $(x,\theta)$. For other $(x,\theta)$ we can then estimate the value of $y_n^\theta(x)$ by the Lipschitz property.

Indeed, by \eqref{eqn:y}, each of $Y(\underline{i}_{n-1},\underline{j}_{n-1};x,\theta)$ is a Lipschitz function (both in $x$ and in $\theta$) with Lipschitz constant not greater than $C_4 \delta^{-1}/2$ for some constant $C_4$ depending only on $p$, $k$ and $\delta$. As every line $\ell_x^\theta$ intersects at most $2k^n$ squares in ${\cal E}_n$, $y_n^\theta$ is also Lipschitz with Lipschitz constant not greater than $C_4 p^{-n} k^n \delta^{-1}$.


Let us for each $n$ choose a sequence $\{\theta_{n,i}\}$ which is $\delta C_4^{-1}
p^{5n/6} k^{-7n/6}$-dense in $[\delta, \pi/2-\delta]$. Similarly,
let us for each $n$ choose a sequence $\{x_{n,i}\}$ which is $\delta C_4^{-1}
p^{5n/6} k^{-7n/6}$-dense in $\Delta$.  We can choose both
sequences with no more than $C_5 \delta^{-1} p^{-5n/6} k^{7n/6}$
elements each. We will denote by $T_{n,j}$ the set of $\theta$ for
which

\[
\forall_{l\neq j} |\theta_{n,l}-\theta| \geq |\theta_{n,j}-\theta|.
\]
Similarly, let $I_{n,i}\subset \Delta$ be defined by

\[
\forall_{l\neq i} |x_{n,l}-x| \geq |x_{n,i}-x|.
\]
The Lipschitz property implies the following lemma.

\begin{lem} \label{lem:lip}
There exists $C_6>0$ such that for every $n>0$, $x\in I_{n,i}$, and $\theta\in T_{n,j}$ we have

\[
\left|y_n^\theta(x) - y_n^{\theta_{n,j}}(x_{n,i}) \right| < C_6(pk)^{-n/6}.
\]
and
\[
\left|y_{n-1}^\theta(x) - y_{n-1}^{\theta_{n,j}}(x_{n,i}) \right| < C_6(pk)^{-n/6}.
\]
\end{lem}

Let $I\subset \Delta$ be an interval of length $k^{-N}$.

\begin{lem} \label{lem:lip2}
There
exist $L',L''>0$ 
such that for any 
\begin{equation} \label{Nreq}
N> 7n/6 + 5n/6 \log_k \frac 1 p + \log_k \frac {C_4} \delta.
\end{equation}
if $I\subset \Delta$ is an interval of length $k^{-N}$ and $n\leq L'N - L''$, then for each $\theta$ the
variation of $y_n^\theta$ inside $I$ is not greater than
$(pk)^{-n/6}$.
\begin{proof}
The variation of a Lipschitz function over an interval is bounded
by the Lipschitz constant times the length of the interval, hence
we only need to know that

\[
k^{-N} < p^{5n/6} k^{-7n/6} \delta C_4^{-1}
\]
which holds by the assumption (\ref{Nreq}).

\end{proof}
\end{lem}

The following part of the proof is similar to the proof of Proposition
\ref{prop:hori}. Corollary \ref{cor:big} does not need any changes, but Corollary \ref{cor:small} will have to be modified.
Let $N_0$ be the smallest number for which

\[
1 + (pk)^{- N_0/3} + 2C_6(pk)^{-N_0/6} < (pk)^{1/8}.
\]

\begin{lem} \label{cor:small2}
If for some $n>N_0$, $j$, and all $x\in I$ $y_n^{\theta_{n,j}}(x) <
(pk)^{n/3}$ then

\[
P\left(\forall m\geq n\ \forall x\in I\ \forall \theta\in T_{n,j} |y_{m+1}^\theta(x) - y_m^\theta(x)|
< (2C_6+1)(pk)^{-m/6}\right) >
\]
\[
1 - C_1 C_5^2  \delta^{-2} p^{-5n/3} k^{7n/3} \sum_{m=n}^\infty \gamma^{(pk)^{m/3}}
\]

\end{lem}
\begin{proof}
The proof is very similar to the proof of Corollary
\ref{cor:small}. In Corollary \ref{cor:small} we divided the $N$-th level $k$-adic interval into $N+1$-st level $k$-adic intervals, and then for each of those intervals we applied Lemma \ref{lem:cre} ii) to prove that $|y_{N+1}^\theta(x) - y_N^\theta|$ is not too large, except when some event of superexponentially small probability happens. As $y_N^\theta$ is piecewise constant, it was enough to check this at just one point from each subinterval. The procedure was then repeated for all the $k^2$ $k$-adic subintervals of level $N+2$ and so on. We got the estimation we were looking for, with a lower bound for the probability that this estimation holds.

Now we do a modified approach (compare \cite{RS}). We divide $I \times T_{n,j}$ into rectangles $I_{n+1,i} \times T_{n+1,l}$. We choose from each of them a point $(\theta_{n,i}, x_{n,l})$ and once again apply Lemma \ref{lem:cre} ii) to prove that $|y_{n+1}^{\theta_{n,i}}(x_{n,l}) - y_N^\theta|$ is not too large, except when some event of superexponentially small probability happens. Knowing the value of $y_{n+1}^\theta(x) - y_n^\theta(x)$ for $\theta = \theta_{n,i}$ and $x=x_{n,l}$ lets us estimate it for all $(\theta, x)\in I_{n+1,i} \times T_{n+1,l}$ by Lemma \ref{lem:lip}:

\[
|y_{m+1}^\theta(x) - y_m^\theta(x)| \leq |y_{m+1}^\theta(x) - y_{m+1}^{\theta_{m+1,j}}(x_{m+1,i})| +
\]
\[
|y_{m+1}^{\theta_{m+1,j}}(x_{m+1,i})- y_m^{\theta_{m+1,j}}(x_{m+1,i})| +
|y_m^{\theta_{m+1,j}}(x_{m+1,i}) - y_m^\theta(x)|
\]

We then repeat the procedure for $n+2,\ldots$. We substitute the estimation for the number of elements of $I_{n,i}$ and $T_{n,l}$. The proof is finished just like in Corollary \ref{cor:small}.
\end{proof}

We are now ready to complete the proof of Proposition \ref{prop:gen}. Let

\begin{equation} \label{eqn:ncond}
N>(LN_0+L'')/L'
\end{equation}
and $n=\lfloor L'N-L''\rfloor$.
We can choose $4k^N$ intervals $I_i\subset \Delta$ of length $k^{-N}$ each in such a way that every pair of points from $\Delta$ in distance no more than $k^{-N-1}$ is contained in one of them.
By Corollary \ref{cor:big} and Lemma \ref{lem:lip2}, for each $i$ we have at least probability

\[
p_N'= 1-C_1 C_5 \delta^{-1} p^{-5n/6} k^{7n/6} \sum_{m=n/L}^{n}  \gamma^{(pk)^{m/3}},
\]
that the function $y_{n}^{\theta_{n,j}}(x)$ is smaller than
$(pk)^{n/4}$ for all $x\in I_i$ and all $\theta_{n,j}$ and its
variation in $I$ is not greater than $(pk)^{-n/6}$. We can then
apply Lemma \ref{cor:small2} to prove that with probability

\[
p_N> 1-C_1 C_5 \delta^{-1} p^{-5n/6} k^{7n/6} \sum_{m=n/L}^{n}  \gamma^{(pk)^{m/3}} - C_1 C_5^2 \delta^{-2} p^{-5n/3} k^{7n/3} \sum_{m=n}^\infty \gamma^{(pk)^{m/3}}
\]
for all $\theta \in [\delta, \pi/2-\delta]$ and $x,y\in I_i$ the limit
$\lim_{m\to\infty} y_m^\theta(x)$ exists and

\[
\left|\lim_{m\to\infty} y_m^\theta(x) - \lim_{m\to\infty} y_m^\theta(y)\right| <
(pk)^{-n/6} + \sum_{m=n}^\infty (2C_6+1) (pk)^{-m/6} =
\]
\[
\left( 1+\frac {2C_6+1} {1-(pk)^{-n/6}}\right) (pk)^{-n/6}.
\]
As $\sum_N 4k^N (1-p_N) < \infty$, by Borel-Cantelli Lemma for every sufficiently large $N$ this holds for all intervals $I_i$. The assertion follows.
\end{proof}
\end{prop}

We can choose in Proposition \ref{prop:gen} $\delta$ arbitrarily close to 0 which will imply the
assertion of Theorem \ref{thm:main} for all $\theta \in (0,\pi/2)$. By
symmetry, the statement holds also for $\theta \in (\pi/2, \pi)$.
The horizontal and vertical projections were considered in the previous section.
This ends the proof of Theorem \ref{thm:main}.

\newpage
\bibliography{ref}

\begin{thebibliography}{WW}


\thispagestyle{empty}


\bibitem[AN]{AN}
K. B. Athreya, P. E. Ney,
\newblock {\it Branching processes}, Springer, New York 1972.

\bibitem[AS]{AS}
N. Alon, J. H. Spencer,
\newblock {\it The Probabilistic Method}, John Wiley and Sons, New York 1992.

\bibitem[C]{C}
L. Chayes,
\newblock {\it Aspects of the fractal percolation process},
\newblock C. Bandt, S. Graf, M. Z¨ahle (eds.), Fractal Geometry and
Stochastics, Birkh¨auser, Basel, 1995.

\bibitem[F]{F}
K. J. Falconer,
\newblock {\it Random fractals},
\newblock Math. Proc. Cambridge Philos. Soc., 100 (1986), 559-582.

\bibitem[FG]{FG}
K. J. Falconer, G.R. Grimmett,
\newblock {\it On the geometry of random Cantor sets
and fractal percolation}
\newblock J. Theoret. Probab. 5 (1992), No.3, 465-485.


\bibitem[G]{G}
G. Grimmett,
\newblock {\it Percolation}, Springer-Verlag, Berlin, 1999.

\bibitem[H]{H}
W. Hoeffding,
\newblock {\it Probability inequalities for sums of bounded random
variables}
\newblock J. Amer. Statist. Assoc. 58 (1963), 13-30.

\bibitem[M]{M}
B. B. Mandelbrot,
\newblock {\it The Fractal Geometry of Nature}, Freeman, San Francisco, 1983.

\bibitem[MW]{MW}
R. D. Mauldin, S. C. Williams,
\newblock {\it Random recursive constructions:
asymptotic geometric and topological properties},
\newblock Trans. Amer. Math. Soc., 295. (1986), 325-346.

\bibitem[RS]{RS}
M. Rams, K. Simon,
\newblock {\it Projections of Fractal Percolations},
\newblock to appear in {\it Ergodic Theory and Dynamical Systems},
\newblock DOI: http://dx.doi.org/10.1017/etds.2013.45.

\end{thebibliography}

\end{document}